\documentclass[reqno]{amsart}
\usepackage{float}
\usepackage{amsmath}
\usepackage{graphicx} 
\usepackage{latexsym}
\usepackage{amsfonts}
\usepackage{amssymb}
\usepackage{color} 
\theoremstyle{plain}
\newtheorem{theorem}{Theorem}
\newtheorem{corollary}[theorem]{Corollary}
\newtheorem{lemma}[theorem]{Lemma}
\newtheorem{proposition}[theorem]{Proposition}
\theoremstyle{definition}

\newtheorem{remark}[theorem]{Remark}
\newtheorem*{remark*}{Remark}

\newcommand{\rd}{\mathbb R^d}
\newcommand{\pr}{\mathbf P}
\newcommand{\e}{\mathbf E}

\begin{document}
\title[Alternative constructions of a harmonic function in a cone]
{Alternative constructions of a harmonic function for a random walk in a cone}
\author[Denisov]{Denis Denisov}
\address{School of Mathematics, University of Manchester, Oxford Road, Manchester M13 9PL, UK}
\email{denis.denisov@manchester.ac.uk}

\author[Wachtel]{Vitali Wachtel}
\address{Institut f\"ur Mathematik, Universit\"at Augsburg, 86135 Augsburg, Germany}
\email{vitali.wachtel@math.uni-augsburg.de}

\begin{abstract}
For a random walk killed at leaving a cone we suggest two new constructions of a positive harmonic
function. These constructions allow one to remove a quite strong extendability assumption, which has
been imposed in our previous paper (Denisov and Wachtel, 2015, Random walks in cones).
As a consequence, all the limit results from that paper remain true for cones which are either convex or star-like and $C^2$. 
\end{abstract}


\maketitle
{\scriptsize
}

\section{Introduction and the main result.}
Consider a random walk $\{S(n),n\geq1\}$ on $\rd$, $d\geq1$, where
$$
S(n)=X(1)+\cdots+X(n)
$$
and $\{X(n), n\geq1\}$ is a family of independent copies of a random
vector $X=(X_1,X_2,\ldots,X_d)$.  We will assume that 
the random variables have zero mean, unit variance,  
and are uncorrelated, that is 
$\e[X_i]=0, \mbox{var}(X_i)=1$ for $1\le i\le d$ and 
$\mbox{cov}(X_i,X_j)=0$ for  $1\le i<j\le d$.

Denote by $\mathbb{S}^{d-1}$ the
unit sphere of $\rd$ and $\Sigma$ an open and connected subset of
$\mathbb{S}^{d-1}$. Let $K$ be the cone generated by the rays
emanating from the origin and passing through $\Sigma$, i.e.
$\Sigma=K\cap \mathbb{S}^{d-1}$.
Let $\tau_x$ be the exit time from $K$ of the random walk with
starting point $x\in K$, that is,
$$
\tau_x=\inf\{n\ge 1: x+S(n)\notin K\}.
$$
In the present paper we are concerned with the 
existence of a positive  harmonic function $V$ for a random walk killed 
at the exit from $K$, that is a function $V$  which solves 
the following equation 
$$
\e[V(x+X),\tau_x>1]=V(x),\quad x\in K. 
$$

Harmonic function $V(x)$ plays a central role in our approach to  
study of the Markov processes confined to unbounded domains. 
This approach was initiated in \cite{DW10}, where 
we studied random walks in a Weyl chamber, which is an example of a cone. 
These studies were extended in \cite{DW15}, where 
we considered random walks in general cones.   
In  particular, in \cite{DW15} we showed that 
$$
\pr(\tau_x>n)\sim C\frac{V(x)}{n^{p/2}},\quad n\to \infty, 
$$
and proved global and local limit theorems for random walks conditioned on $\{\tau_x>n\}$.
The approach suggested in \cite{DW15} was further extended to one-dimensional random
walks above the curved boundaries~\cite{DW16}, \cite{DSW16}, \cite{DSW18}, integrated
random walks~\cite{DW15b}, \cite{DKW}, products of random matrices~\cite{GLP16}, and
Markov walks~\cite{GLL16}. 

This approach is based on the universality ideas and heavily relies on corresponding
results for Brownian motion, or, more generally, diffusion processes. Thus, an
important role is  played by the harmonic function of the Brownian motion killed at the
boundary of $K$, which  can be described as the minimal (up to a constant), strictly positive
on $K$ solution of the following boundary problem:
$$
\Delta u(x)=0,\ x \in K\quad\text{with boundary condition }u\big|_{\partial
  K}=0.
$$
The function $u(x)$ and constant $p$ can be found as follows.
If $d=1$ then we have only one non-trivial cone $K=(0,\infty)$.
In this case $u(x)=x$ and $p=1$. Assume now that $d\geq2$.
Let $L_{\mathbb{S}^{d-1}}$ be the Laplace-Beltrami operator on
$\mathbb{S}^{d-1}$ and assume that $\Sigma$ is regular with respect to $L_{\mathbb{S}^{d-1}}$.
With this assumption, there exists a complete set of orthonormal eigenfunctions
$m_j$  and corresponding eigenvalues $0<\lambda_1<\lambda_2\le\lambda_3\le\ldots$ satisfying
\begin{align}
 \label{eq.eigen}
  L_{\mathbb{S}^{d-1}}m_j(x)&=-\lambda_jm_j(x),\quad x\in \Sigma\\
  \nonumber m_j(x)&=0, \quad x\in \partial \Sigma .
\end{align}
Then
$$
p=\sqrt{\lambda_1+(d/2-1)^2}-(d/2-1)>0.
$$
and the harmonic function $u(x)$ of the Brownian motion is given by
\begin{equation}
\label{u.from.m}
u(x)=|x|^pm_1\left(\frac{x}{|x|}\right),\quad x\in K.
\end{equation}
We refer to \cite{BS97} for further details on exit times of Brownian motion. 
For symmetric stable L\'evy processes asymptotics for exit times and 
related questions have been considered in \cite{BB04}, \cite{BPW18} and \cite{KRS18}, 
see also references therein.

In \cite{DW15} we showed that one construct a harmonic function for the random walk killed at $\tau_x$ as 
follows
$$
V(x)=\lim_{n\to \infty} \e[u(x+S(n),\tau_x>n].
$$
The existence and positivity of $V$ was shown under certain assumptions. 
The geometric assumptions in \cite{DW15} can be summarised as follows, 
\begin{enumerate}
  \item $K$ is either starlike with $\Sigma$ in $C^2$ or 
  convex. We say that $K$ is starlike if there exists $x_0\in \Sigma$ such that
    $x_0+K\subset K$ and ${\rm dist}(x_0+K, \partial K)>0$.
    Clearly, every convex cone is also starlike, for the proof see Remark 15 in \cite{DW15}.

  \item We assume that there exists an open and connected set
    $\widetilde \Sigma\subset \mathbb{S}^{d-1}$ with ${\rm
        dist}(\partial \Sigma, \partial \widetilde \Sigma)>0$ such that
    $\Sigma\subset \widetilde \Sigma$ and the
    function $m_1$  can be extended to $\widetilde \Sigma$ as a solution to
    (\ref{eq.eigen}).
\end{enumerate}
  
Assumption (ii) is quite restrictive.  For this assumption to hold it is necessary
to assume that the boundary of the cone is piecewise infinitely differentiable. But
this condition is not sufficient. The restriction (ii) excludes many cones which are 
of interest in various mathematical problems. For example, it is not clear whether 
(ii) holds for linear transformations of the orthant $\mathbb{R}_+^d$, $d\ge 2$ which
appear often in paths enumeration problems in combinatorics. (It is worth mentioning that
(ii) holds for any simply connected open cone in $\mathbb{R}^2$. This follows from the 
observation that $m_1(x)=\sin(C_1+C_2x)$ in this two-dimensional situation.) 

We have shown in \cite{DW15} that the condition (ii) can be dropped in the case when the 
random walk $\{S(n)\}$ has bounded jumps, Raschel and Tarrago \cite{RT18} have recently
shown that (ii) can be removed under  stronger than in \cite{DW15} moment restrictions on 
the vector $X$. The {\it main aim} of this paper is to show that this assumption can be
removed without imposing any further conditions. Namely, we prove that (i) is sufficient
and the following result holds  
\begin{theorem}
  \label{thm:C2}
  Assume that either the cone $K$ is convex or $\Sigma$ is $C^2$ and $K$ is starlike.
  If $\mathbf{E}|X|^{\alpha}$ is finite for $\alpha = p$ if $p > 2$ or for some $\alpha > 2$ if $p \le 2$, then the
  function 
  $$
  V(x):=\lim_{n\to\infty}\mathbf{E}\left[u(x+S(n));\tau_x>n\right]
  $$
  is finite and harmonic for $\{S(n)\}$ killed at leaving $K$, i.e.,
  $$
  V(x)=\mathbf{E}\left[V(x+S(n));\tau_x>n\right],\quad x\in K,\ n\ge1.
  $$
  Furthermore, $V(x)$ is strictly positive on the set 
  \begin{align*}
K_+:=&\left\{x\in K:\text{ there exists }\gamma>0\text{ such that for every } R>0\right.\\
&\left.\text{ there exists }n\text{ such that }
\mathbf{P}(x+S(n)\in D_{R,\gamma},\tau_x>n)>0\right\},
\end{align*}
where
$D_{R,\gamma}:=\{x\in K:|x|\geq R, {\rm dist}(x,\partial K)\geq\gamma|x|\}$.
  \end{theorem}
  
We will present two very different proofs of this theorem. 
The first proof uses preliminary bounds for the moments of exit times of $\tau_x$
due to \cite{MC84}, see Lemma~\ref{moments} below. The proof is similar to that in
\cite{DW15}, but we use an additional idea of time-dependent shifts inside the cone.
Thus the approach is reminiscent of one-dimensional random walks conditioned to stay
above curved boundaries~\cite{DSW16}. 

The second proof combines time-dependent shifts with an iterative procedure similar to 
that in \cite{DW10} and \cite{DW15b}. The main advantage of this approach is that in
principle no preliminary information on moments of exit times is needed. However, we use
\cite{MC84} to obtain optimal moment conditions. If we assume two additional moments then
this approach becomes self-contained, see Remark~\ref{without_mc} below. 

A further advantage of new constructions consists in the fact that we do not use estimates
for the concentration function of the random walk $\{S(n)\}$, which were important for the 
method used in \cite{DW15}.

Since the geometric assumption (ii) has been used in \cite{DW15} in the construction of
$V(x)$ only, Theorem~\ref{thm:C2} allows us to state limit theorems for random walks in 
cones proven in \cite{DW15} and in \cite{DuW15} for all cones satisfying (i).
\begin{corollary}
\label{cor:int.lim}
Under the conditions of Theorem~\ref{thm:C2}, as $n\to\infty$,
\begin{equation*}
\mathbf{P}(\tau_x>n)\sim \varkappa V(x)n^{-p/2},
\end{equation*}
\begin{equation*}
\mathbf{P}\left(\frac{x+S(n)}{\sqrt{n}}\in\cdot\big|\tau_x>n\right)\to\mu\quad\text{weakly},
\end{equation*}
where $\mu$ is a probability measure on $K$ with the density $H_0u(y)e^{-|y|^2/2}$. Furthermore,
the process $\left\{\frac{x+S([nt])}{\sqrt{n}},\ t\in[0,1]\right\}$ conditioned on $\{\tau_x>n\}$
converges weakly in the space $D([0,1],\|\cdot\|_\infty)$.
\end{corollary}
\begin{corollary}
Assume that $X$ takes values on a lattice $R$ which is a non-degenerate linear transformation 
of $\mathbb{Z}^d$. Then, under the assumptions of Theorem~\ref{thm:C2},
$$
\sup_{y\in D_n(x)}\left|n^{p/2+d/2}\mathbf{P}\left(x+S(n)=y,\tau_x>n\right)-
      C_0 V(x) u\left(\frac{y}{\sqrt{n}}\right)e^{-|y|^2/2n}\right|\to 0,
$$
where
$$
D_n(x):=\{y\in K:\,\mathbf{P}(x+S(n)=y)>0\}.
$$
The constant $C_0$ is a product of the volume of the unit cell in $R$ and of a factor, which
depends on the periodicity of the distribution of $X$.
\end{corollary}

In the proof of Theorem 5 in \cite{DW15} we have required the strong aperiodicity of $X$. 
This has been done to use the simplest version of the local limit theorem for unrestricted
random walks from Spitzer's book \cite{Sp76}.  But this standard result can be replaced by Stone's local
limit theorem which is valid for all lattice walks, see \cite{Stone67}.
\section{Preliminary estimates}
We first collect some useful facts about the classical harmonic function $u(x)$.
\begin{lemma}
\label{lem:harnack}
There exists a constant $C=C(d)$ such that for $x\in K$
\begin{align}
\left|\nabla u(x)\right|\le C\frac{u(x)}{{\rm dist}(x,\partial K)},\nonumber\\
\left|u_{x_i}\right|\le C\frac{u(x)}{{\rm dist}(x,\partial K)},\nonumber\\
\left|u_{x_ix_j}\right|\le C\frac{u(x)}{{\rm dist}(x,\partial K)^2},\nonumber\\
\left|u_{x_ix_j x_k}\right|\le C\frac{u(x)}{{\rm dist}(x,\partial K)^3}.
\label{eq:bound.u}
\end{align}

\end{lemma}
\begin{proof}
Recalling that every partial derivative $u_{x_i}$ is harmonic and using the mean value theorem for
harmonic functions, we obtain
$$
u_{x_i}(x)=\frac{1}{{\rm Vol}(B(x,r))}\int_{B(x,r)}u_{x_i}(y)dy,
$$
where $B(x,r)$ is the ball of radius $r$ around $x$ and $r<{\rm dist}(x,\partial K)$.
By the Gauss-Green theorem,
$$
u_{x_i}(x)=\frac{1}{{\rm Vol}(B(x,r))}\int_{\partial B(x,r)}u(z)(\nu(z), e_i)dz,
$$
where $\nu(z)$ is the outer normal at $z$. Choosing $r={\rm dist}(x,\partial K)/2$ and applying
the Harnack inequality in the ball $B(x,{\rm dist}(x,\partial K))$, we conclude that
$$
|u_{x_i}(x)|\le 3\cdot2^{d-2}\frac{{\rm Vol}(\partial B(x,r))}{{\rm Vol}(B(x,r))}u(x)
=3d \cdot  2^{d-1}\frac{u(x)}{{\rm dist}(x,\partial K)}.
$$
This implies the desired estimate for $u_{x_i}(x)$. Since 
$u_{x_j}$ is harmonic as well 
we can write 
\begin{align*}
|u_{x_ix_j}|&=
\left|
\frac{1}{{\rm Vol}(B(x,r))}\int_{\partial B(x,r)}
u_{z_i}(z)(\nu(z), e_j)dz
\right|\\ 
&\le \frac{C(d)}{{\rm Vol}(B(x,r)){\rm dist}(x,\partial K)}
\int_{\partial B(x,r)}
u(z)|(\nu(z), e_j)|dz
\le 
C(d)^2\frac{u(x)}{{\rm dist}(x,\partial K)^2}
\end{align*}
The inequality for the third derivative can be  proved analogously. The inequality for the gradient immediately follows from the inequality
for the first derivative. 
\end{proof}

For every cone $K$ one has the bound
$$
{\rm dist}(x,\partial K)\le |x|,\quad x\in K.
$$
Furthermore, it follows from \eqref{u.from.m} that
$$
u(x)\le C|x|^p,\quad x\in K. 
$$
In the next lemma we derive more accurate estimates for $u(x)$.
\begin{lemma}
\label{lem:u-prop}
Assume that either the cone $K$ is convex or $\Sigma$ is $C^2$ and $K$ is starlike. Then
\begin{equation}
\label{u-bounds}
C_1\left({\rm dist}(x,\partial K)\right)^p
\le u(x)\le 
C_2|x|^{p-1}{\rm dist}(x,\partial K),\quad x\in K
\end{equation}
and
\begin{equation}
\label{u'-bound}
|\nabla u(x)|\le C_3|x|^{p-1},\quad x\in K.
\end{equation}
\end{lemma}

\begin{proof}
The upper bound in \eqref{u-bounds} is (0.2.3) in Varopoulos \cite{Var99}
and the lower bound has been proved in Lemma~19 in \cite{DW15}.
Combining the upper bound in \eqref{u-bounds} with Lemma~\ref{lem:harnack}, we obtain \eqref{u'-bound}.
\end{proof}
We will extend  the function $u$ by putting $u(x)=0$ for $x\notin K$. 
\begin{lemma}
  \label{lem:u-diff}
  Assume that either the cone $K$ is convex or $\Sigma$ is $C^2$ and $K$ is starlike. 
  Let $x\in K$. Then, 
  \begin{equation}
    \label{diff-bound}
    |u(x+y)-u(x)|\le  C|y|\left(|x|^{p-1}+|y|^{p-1}\right)
  \end{equation}
and, for $|y|\le |x|/2$,
\begin{equation}
    \label{diff-bound1}
    |u(x+y)-u(x)|\le  C|y||x|^{p-1}.
  \end{equation}
For $p<1$ and $x\in K$, 
\begin{equation}
  \label{diff-bound2}
  |u(x+y)-u(x)|\le  C|y|^{p}. 
\end{equation}
 
\end{lemma}  

\begin{proof}
Consider first the case $p\ge 1$.     
To prove \eqref{diff-bound} consider first the case when the interval $[x,x+y]$ lies in $K$. 
Then, 
$$
|u(x+y)-u(x)|=\left|\int_0^1(\nabla u(x+ty),y)dt\right|\le |y|\int_0^1|\nabla u(x+ty)|dt.
$$
Hence,  by \eqref{u'-bound},
$$
|u(x+y)-u(x)|\le C_3|y|\int_0^1|x+ty|^{p-1}dt\le C2^{p-1}|y|\big(|x|^{p-1}+|y|^{p-1}\big),
$$
as required. 
Now if $[x,x+y]$ does not belong to $K$ then 
we have two cases:  $x+y\in K$ or $x+y\notin K$. 
If $x+y\in K$ then 
there exist $t_1,t_2:0<t_1<t_2<1$ such that 
$[x,x+t_1y)\subset K$ and $(x+t_2y,x+y]\subset K$ and $x+t_1y,x+t_2y\in \partial K$. 
If $x+y\notin K$ then 
there exists $t_1: 0<t_1$ such that 
$[x,x+t_1y)\subset K$ and we put $t_2=1$. 
Since in both cases $x+t_1y,x+t_2y\notin K$ and $u=0$ outside  of $K$ we obtain 
\begin{align*}
|u(x)-u(x+y)|&=|u(x)-u(x+t_1y)+u(x+t_2y)-u(x+y)|\\
&\le |u(x)-u(x+t_1y)|+|u(x+t_2y)-u(x+y)|\\
&=\left|\int_0^{t_1}(\nabla u(x+ty),y)dt\right|
+\left|\int_{t_2}^1(\nabla u(x+ty),y)dt\right|\\
\mbox{(by \eqref{u'-bound})}&\le C_4|y|\left(\int_0^{t_1}+\int_{t_2}^1\right) 
|x+ty|^{p-1}dt\le C|y|\big(|x|^{p-1}+|y|^{p-1}\big),
\end{align*}
as required. If $p\ge1$ then \eqref{diff-bound1} is immediate from
\eqref{diff-bound}. 

For $p<1$ we will prove a stronger statement \eqref{diff-bound2} which clearly implies
\eqref{diff-bound}. Consider first again the case when the interval $[x,x+y]$ lies in $K$. 
If $|x|\ge 2|y|$ then 
\begin{align*}
  |u(x+y)-u(x)|&\le |y|\int_0^1|\nabla u(x+ty)|dt\\
      &\le C |y|\int_0^1|x+ty|^{p-1} dt 
      \le C|y| (2|y|-|y|)^{p-1}
      \le C |y|^p
\end{align*}
and
\begin{align*}
  |u(x+y)-u(x)|&\le |y|\int_0^1|\nabla u(x+ty)|dt\\
      &\le C |y|\int_0^1|x+ty|^{p-1} dt 
      \le C|y| (|x|-|x|/2)^{p-1}
      \le C |y||x|^{p-1}.
\end{align*}
Therefore, we have \eqref{diff-bound1} and \eqref{diff-bound2} for
$|y|\le |x|/2$. Furthermore, for $|x|< 2|y|$ one has 
\begin{align*}
    |u(x+y)-u(x)|&\le 
    C(|x+y|^p+|y|^p) \le C(3^p+1)|y|^p,
\end{align*}
which completes the proof \eqref{diff-bound2} in the case when $[x,x+y]\subset K$.
The case when $[x,x+y]$ does not belong to $K$ can be considered in the same way as for $p\ge1$. 
\end{proof}

For $x\in K$ let 
\begin{equation}\label{eq:defn.f}
f(x) = \e[u(x+X)] -u(x).
\end{equation}
Next we require a bound on $f(x)$.
\begin{lemma}
  \label{lem:bound.f}
  Let the assumptions of Theorem~\ref{thm:C2} hold and $f$ be
  defined by \eqref{eq:defn.f}.   Then, for some $\delta>0$,
  \begin{equation*}
    |f(x)| \le C \frac{|x|^{p}}{{\rm dist}(x,\partial K)^{2+\delta}}\quad\text{for all } x\in K\text{ with }|x|\geq 1.
  \end{equation*}
Furthermore,
$$
|f(x)|\leq C \quad\text{for all } x\in K\text{ with }|x|\leq 1.
$$
\end{lemma}
\begin{proof}
  Let $x\in K$ be such that $|x|\geq 1$.  
  Put $g(x)={\rm dist}(x,\partial K),$ 
  and let $\eta \in (0,1)$. 
Then, for any $y\in B(0,\eta g(x))$, the interval $[x,x+y]\subset K$.
By the Taylor theorem,
  \begin{align*}
    \left|u(x+y)-u(x)-\nabla u \cdot y
      -\frac{1}{2}\sum_{i,j}u_{x_ix_j}y_iy_j\right| \le R_3(x)|y|^3.
  \end{align*}
  The remainder $R_3(x)$ can
  be estimated by Lemma~\ref{lem:harnack},
$$
R_3(x)=C_d\max_{z\in B(x,\eta g(x))}\max_{i,j,k}|u_{x_ix_jx_k}(z)|
\le C\frac{(1+\eta)^{p}}{(1-\eta)^3}\frac{|x|^{p}}{g(x)^3},
$$
which will give us
\begin{equation}
  \label{eq:taylor}
  \left|u(x+y)-u(x)-\nabla u \cdot y -\frac{1}{2}\sum_{i, j}u_{x_ix_j} y_iy_j\right|
  \le C\frac{|x|^{p}}{g(x)^{3}}|y|^3.
\end{equation}
Then  we can proceed as follows
\begin{align*}
  |f(x)|
  &=|\e \left(u(x+X)-u(x)\right) {\bf 1} (|X| \le \eta g(x)) | \\
  &\hspace{1cm}+|\e \left(u(x+X)-u(x)\right) {\bf 1} (|X|>\eta g(x)) | \\
  &\le\left| \e \left[\left(\nabla u\cdot X +\frac{1}{2}\sum_{i,j}u_{x_ix_j}X_iX_j\right){\bf 1}(|X|\le \eta g(x))\right]\right|\\
  &\hspace{1cm}+C\frac{|x|^{p}}{g(x)^{3}}\e \left[|X|^3{\bf 1}(|X|\le \eta g(x))\right]
  \\
  &\hspace{1cm} +C\e\left[(|x|^p+|X|^p){\bf 1}(|X|>\eta g(x))\right].
\end{align*}
Here we used also the bounds $|u(x+y)-u(x)|\leq C(|x+y|^p+|x|^p)\leq C(|x|^p+|y|^p)$ 
valid for all $x$ and $y$ .
After rearranging the terms we obtain 
\begin{align*}
  |f(x)| &\le
  \left |\e \left[\nabla u\cdot X +\frac{1}{2}\sum_{i,j}u_{x_ix_j}X_iX_j\right]\right| \\
  &\hspace{1cm}+\left|\e\left[\left(\nabla u\cdot X +\frac{1}{2}\sum_{i,j}u_{x_ix_j}X_iX_j\right){\bf 1}(|X|>\eta g(x))\right]\right|\\
  &\hspace{1cm}+C\frac{|x|^{p}}{g(x)^{3}} \e \left[|X|^3{\bf 1}(|X|\le \eta g(x))\right]\\
  &\hspace{1cm}+C\e\left[(|x|^p+|X|^p){\bf 1}(|X|>\eta g(x))\right].
\end{align*}
Now note that the first term is $0$ due to $\mathbf EX_i=0$, $\mbox{cov}(X_i,
X_j)=\delta_{ij}$ and $\Delta u=0$.  The partial derivatives of the function $u$
in the second term can be  estimated via Lemma~\ref{lem:harnack}, which results in the following estimate 
\begin{align*}
  |f(x)|&\le C\biggl(\frac{|x|^{p}}{g(x)} \e \left[|X|;|X|>\eta
    g(x)\right]
  +\frac{|x|^{p}}{g(x)^{2}} \e \left[|X|^2;|X|>\eta g(x)\right]\\
  &\hspace{1cm}+\frac{|x|^{p}}{g(x)^{3}} \e \left[|X|^3;|X|\le \eta
    g(x)\right]
  +|x|^{p}\mathbf P(|X|>\eta g(x))\\
  &\hspace{1cm}+\mathbf E \left[|X|^p;|X|> \eta g(x)\right]
  \biggr).
\end{align*}
Hence, from the Markov inequality we conclude
\begin{align}
\label{new1}
\nonumber
  |f(x)|&\le
  C\frac{|x|^{p}}{\eta^2 g^2(x)}\mathbf E \left[|X|^2;|X|>\eta g(x)\right]
  +C\frac{|x|^{p}}{ g^3(x)}\mathbf E \left[|X|^3;|X|\le \eta g(x)\right]\\
&\hspace{1cm}+C\mathbf E \left[|X|^p;|X|> \eta g(x)\right].
\end{align}
Now recall the moment assumption that $\mathbf
E|X|^{2+\delta}<\infty$ for some $\delta>0$.  The first term is
estimated via the Chebyshev inequality,
$$
\frac{|x|^{p}}{\eta^2 g^2(x)}\mathbf E
\left[|X|^2;|X|>\eta g(x)\right]\leq
\frac{|x|^{p}}{\eta^{2+\delta} g^{2+\delta}(x)}\mathbf E
|X|^{2+\delta}.
$$
The second term can be estimated similarly,
\begin{equation*}
  \frac{|x|^{p}}{g^3(x)}\mathbf E \left[|X|^3;|X|\le \eta g(x)\right]\leq
  \frac{|x|^{p}}{\eta^2 g^3(x)}\eta^{1-\delta}g^{1-\delta}(x)\mathbf E |X|^{2+\delta}.
\end{equation*}
 In order to bound the last term
in (\ref{new1}) we have to distinguish between $p\leq2$ and $p>2$.

If $p\leq2$, then, by the Chebyshev inequality,
$$
\mathbf{E}\left[|X|^p;|X|>\eta g(x)\right]\leq
\frac{1}{(\eta g(x))^{2+\delta-p}}\mathbf{E}\left[|X|^{2+\delta}\right]
\leq C\frac{|x|^{p}}{g^{2+\delta}(x)},
$$
as $g(x) = {\rm dist}(x,\partial K)\le |x|.$

In case $p>2$ we have, according to our moment condition, $\mathbf{E}[|X|^p]<\infty$.
Consequently,
$$
\mathbf{E}\left[|X|^p;|X|>\eta g(x)\right]\leq C.
$$
The second statement follows easily from the fact that $u(x)$ is bounded on $|x|\leq1$ and
the inequality $\mathbf{E}[u(x+X)]\leq C(1+\mathbf{E}[|X|^p])$.
\end{proof}
We derive next an estimate for the maximum 
$$
M(n):=\max_{k\le n}|S(k)|,
$$
which will be used several times in the proofs of our main results.
\begin{lemma}
\label{lem:tail}
If $\mathbf{E}|X|^t<\infty$ for some $t\ge2$ then, uniformly in $x$, as $n\to\infty$,
$$
\mathbf{E}\left[M^t(n);\tau_x>n,M(n)>n^{1/2+\varepsilon/2}\right]=o\left(\mathbf{E}[\tau_x\wedge n]\right).
$$
\end{lemma}
\begin{proof}
For every fixed $a>0$ one has
\begin{align}
\label{nu.3}
\nonumber
&\mathbf{P}\left(M(n)>r,\tau_x>n\right)\\
&\hspace{1cm}\le
\mathbf{P}\left(M(n)>r,\max_{j\le n}|X(j)|\le ar\right)+
\mathbf{P}\left(\max_{j\le n}|X(j)|>ar,\tau_x>n\right).
\end{align}
Using first the standard union bound and then the Fuk-Nagaev-type inequality from Corollary 23 in
\cite{DW15}, one gets
\begin{align}
\label{nu.4}
\mathbf{P}\left(M(n)>r,\max_{j\le n}|X(j)|\le ar\right)
\le 2dn \left(\frac{\sqrt{d}e}{a}\right)^{1/a}\left(\frac{n}{r^2}\right)^{1/(a\sqrt d)}.
\end{align}
Furthermore,
\begin{align}
\label{nu.5}
\nonumber
\mathbf{P}\left(\max_{j\le n}|X(j)|>ar,\tau_x>n\right)
&\le\sum_{j=1}^n\mathbf{P}\left(|X(j)|>ar,\tau_x>n\right)\\
\nonumber
&\le\sum_{j=1}^n\mathbf{P}\left(|X(j)|>ar,\tau_x>j-1\right)\\
&=\mathbf{E}[\tau_x\wedge n]\mathbf{P}(|X|>ar).
\end{align}
Combining \eqref{nu.3}--\eqref{nu.5}, we conclude that
\begin{align*}
\mathbf{P}\left(M(n)>r,\tau_x>n\right)\le
2dn \left(\frac{\sqrt{d}e}{a}\right)^{1/a}\left(\frac{n}{r^2}\right)^{1/(a\sqrt d)}
+\mathbf{E}[\tau_x\wedge n]\mathbf{P}(|X|>ar).
\end{align*}
Choosing here $a=\frac{2\varepsilon}{\sqrt d ((1+\varepsilon)t+5)}$  
and integrating the latter bound, one easily gets the bound
\begin{align}
  &\mathbf{E}\left[(M(n))^{t};\tau_x>n, M(n)> n^{1/2+\varepsilon/2}\right]
  \nonumber\\
  &\hspace{1cm}\le C(a)\left(n^{-3/2}+\mathbf{E}[\tau_x\wedge n]\mathbf{E}\left[|X|^t;|X|>an^{1/2+\varepsilon/2}\right]\right)
    \label{eq.max.nag}.
\end{align}
Thus, the proof is complete. 
\end{proof}

Finally, we will require  the following results from \cite{MC84}.   
\begin{lemma}
  \label{moments}
  For every $\beta<p$ we have
  \begin{equation}
    \label{MC}
    \mathbf{E}[\tau_x^{\beta/2}]\leq C(1+|x|^\beta)
  \end{equation}
  and
  \begin{equation}
    \label{max-mom}
    \mathbf{E}[M^\beta(\tau_x)]\leq C(1+|x|^\beta),
  \end{equation}
  where $M(\tau_x):=\max_{k\leq\tau_x}|x+S(k)|$.
\end{lemma}
This is the statement of Theorem 3.1 of \cite{MC84}. One has only to notice that
$e(\Gamma,R)$ in that theorem is denoted by $p$ in our paper.

\section{First proof of Theorem~\ref{thm:C2}}
Since $K$ is starlike there exists  $x_0\in K$ with $|x_0|=1$, $x_0+K\subset K$ and $R_0$
such that ${\rm dist}(R_0x_0+K,\partial K)>1$.
For $k\ge 0$ set 
$$
g_k= k^{1/2-\gamma}R_0x_0,
$$
where $\gamma\in(0,\min(1/2,p))$. 
First we will show that  it is sufficient to show convergence of 
$$
\e[u(x+g_k+S(k));\tau_x>k] 
$$
as $k$ to infinity. 

\begin{lemma}
  For any $x\in K$, as $k\to\infty$,
\begin{equation}
  \label{eq.harmonic-app}
  \e[u(x+g_k+S(k));\tau_x>k] - \e[u(x+S(k));\tau_x>k]\to 0.
\end{equation}
\end{lemma}  

\begin{proof}
 Consider first the case $p\ge 1$.    
 Using \eqref{diff-bound}, we obtain
\begin{align}
\label{new.1}
\nonumber
&\big|\e[u(x+g_k+S(k));\tau_x>k] - \e[u(x+S(k));\tau_x>k]\big|\\ 
\nonumber
&\hspace{1cm}=\left|\e[u(x+g_k+S(k))-u(x+S(k));\tau_x>k] \right|\\ 
\nonumber
&\hspace{1cm}\le C|g_k| \e[|x+S(k)|^{p-1};\tau_x>k] +C |g_k|^p\pr(\tau_x>k)\\
&\hspace{1cm}\le C|g_k| \e[|S(k)|^{p-1};\tau_x>k] +C(1+|x|^{p-1}) |g_k|^p\pr(\tau_x>k).
\end{align}  
Using the Markov inequality and \eqref{MC} with $\beta=p-p\gamma$, we get
\begin{equation}
\label{new.2}
|g_k|^p\pr(\tau_x>k) \le C k^{p/2-p\gamma} \frac{\e[\tau_x^{p/2-p\gamma/2}]}{k^{p/2-p\gamma/2}} \to 0, \quad k\to \infty. 
\end{equation}
Furthermore,
\begin{align}
\label{new.2a}
\nonumber
&\e[|S(k)|^{p-1};\tau_x>k]\\
\nonumber
&\hspace{1cm}\le k^{(1+\varepsilon)(p-1)/2}\mathbf{P}(\tau_x>k)
+\e[|S(k)|^{p-1};\tau_x>k, |S(k)|>k^{(1+\varepsilon)/2}]\\
&\hspace{1cm}\le k^{(1+\varepsilon)(p-1)/2}\mathbf{P}(\tau_x>k)
+k^{-(1+\varepsilon)/2}\e[|S(k)|^{p};\tau_x>k, |S(k)|>k^{(1+\varepsilon)/2}].
\end{align}
Choosing $\varepsilon<\gamma/(p-1)$, 
applying the Markov inequality  and using \eqref{MC} with $\beta=p-\varepsilon(p-1)$,
we conclude that
\begin{equation}
\label{new.3}
|g_k|k^{(1+\varepsilon)(p-1)/2}\mathbf{P}(\tau_x>k)
\le  
|g_k|k^{(1+\varepsilon)(p-1)/2} \frac{\e[\tau_x^{p/2-\varepsilon(p-1)/2}]}
{k^{p/2-\varepsilon(p-1)/2}}
\to 0.
\end{equation}
If $p>2$ then $\mathbf{E}\tau_x$ is finite and, by Lemma~\ref{lem:tail},
\begin{equation}
\label{new.4}
|g_k|k^{-(1+\varepsilon)/2}\e[|S(k)|^{p};\tau_x>k, |S(k)|>k^{(1+\varepsilon)/2}]\to0.
\end{equation}
If $p\le2$ then, using \eqref{MC} once again, we have
$$
\mathbf{E}[\tau_x\wedge k]\le k^{1-p/2+\delta/2}\mathbf{E}[\tau_x^{p/2-\delta/2}]
\le C(1+|x|^p)k^{1-p/2+\delta/2}.
$$
Combining this estimate with Lemma~\ref{lem:tail}, we obtain
\begin{align*}
&\mathbf{E}\left[|S(k)|^p;\tau_x>k,|S(k)|>k^{(1+\varepsilon)/2}\right]\\
&\hspace{1cm}
\le k^{-(2+\delta-p)(1+\varepsilon)/2}\mathbf{E}\left[|S(k)|^{2+\delta};\tau_x>k,|S(k)|>k^{(1+\varepsilon)/2}\right]
\to 0.
\end{align*}
Therefore, \eqref{new.4} remains valid for $p\le2$.
Combining \eqref{new.3} and \eqref{new.4}, we conclude that
$$
|g_k|\e[|S(k)|^{p-1};\tau_x>k]\to0.
$$
Applying this and \eqref{new.2} to the right hand side in \eqref{new.1}, we have
\eqref{eq.harmonic-app}.

We are  left to  consider the case $p<1$. By \eqref{diff-bound2}, 
we immediately arrive at   
\begin{align*}
  \left|\e[u(x+g_k+S(k))-u(x+S(k));\tau_x>k ]\right| 
  &\le C |g_k|^p\mathbf{P}(\tau_x>k) \\ 
  &\le C|g_k|^p \frac{\e[\tau_x^{p/2-p\gamma/2}]}{k^{p/2-p\gamma/2}}
  \to 0. 
\end{align*}   
 
\end{proof}
Now we prove the existence of the limit of the sequence $\mathbf{E}[u(x+g_k+S(k));\tau_x>k]$. 
\begin{proposition}
\label{prop:main}
There exist a finite  function $V(x)$ such that 
$$
\lim_{k\to\infty} \e[u(x+g_k+S(k));\tau_x>k] = V(x). 
$$
\end{proposition}
We shall split the proof of this proposition into several steps.
To this end we shall use the following decomposition:
\begin{align}
\label{decomp_main}
\nonumber
&u(x+g_k+S(k))\mathbb{I}\{\tau_x>k\}\\
\nonumber
&\hspace{0.5cm}=u(x)+\sum_{l=1}^k\left[u(x+g_l+S(l))\mathbb{I}\{\tau_x>l\}
-u(x+g_{l-1}+S(l-1))\mathbb{I}\{\tau_x>l-1\}\right]\\
\nonumber
&\hspace{0.5cm}=u(x)-\sum_{l=1}^k u(x+g_l+S(l))\mathbb{I}\{\tau_x=l\}\\
\nonumber 
&\hspace{1.5cm}+\sum_{l=1}^k\left[u(x+g_l+S(l))-u(x+g_{l-1}+S(l-1))\right]\mathbb{I}\{\tau_x>l-1\}\\
\nonumber
&\hspace{0.5cm}=u(x)-u(x+g_{\tau_x}+S(\tau_x))\mathbb{I}\{\tau_x\le k\}\\
\nonumber 
&\hspace{1.5cm}+\sum_{l=1}^k\left[u(x+g_l+S(l))-u(x+g_{l-1}+S(l))\right]\mathbb{I}\{\tau_x>l-1\}\\
\nonumber
&\hspace{1.5cm}+\sum_{l=1}^k\left[u(x+g_{l-1}+S(l))-u(x+g_{l-1}+S(l-1))\right]\mathbb{I}\{\tau_x>l-1\}\\
&\hspace{0.5cm}=:u(x)-W_k^{(1)}(x)+W_k^{(2)}(x)+W_k^{(3)}(x).
\end{align}
The proposition will follow if we show that the expectations of all three random
variables in \eqref{decomp_main} converge, as $k\to\infty$, to finite limits.
\begin{lemma}
\label{lem:W1}
The sequence $W_k^{(1)}(x)$ converges almost surely and in $L^1$
towards $u(x+g_{\tau_x}+S(\tau_x))$. Furthermore,
\begin{equation}
\label{W1}
\mathbf{E}u(x+g_{\tau_x}+S(\tau_x))\le C(1+|x|^{p-\gamma}).
\end{equation}
\end{lemma}
\begin{proof}
The almost sure convergence is immediate from the fact that the sequence
$W_k^{(1)}$ is increasing. Thus, it remains to show that \eqref{W1} holds.

Since $x+S(\tau_x)\notin K$,
${\rm dist}(x+g_{\tau_x}+S(\tau_x),\partial K)\le |g_{\tau_x}|$
in the case when $x+g_{\tau_x}+S(\tau_x)\in K$.

Assume first that $p<1$. Combining \eqref{diff-bound2} and \eqref{MC},
we obtain
$$
\e[u(x+g_{\tau_x}+S(\tau_x))] \le 
C\e[|g_{\tau_x}|^p]
\le C \e[\tau_x^{p/2-p\gamma}]
\le  C(1+|x|^{p-\gamma}).
$$

Consider now the case $p\ge1$.
Then, using the upper bound \eqref{u-bounds}, we obtain  
\begin{align*}
\e[u(x+g_{\tau_x}+S(\tau_x))]
&\le C\e[|g_{\tau_x}|(|x|+|g_{\tau_x}|+|S(\tau_x)|))^{p-1}]\\
&\le C
\left(
  |x|^{p-1}\e|g_{\tau_x}|+\e|g_{\tau_x}|^p+  \e[|g_{\tau_x}|M(\tau_x)^{p-1}]
\right).
\end{align*}  
Recalling the definition of the sequence $g_k$, we have
$$
\e[u(x+g_{\tau_x}+S(\tau_x))]
\le C|x|^{p-1}\e\tau_x^{1/2-\gamma}
+C\e\tau_x^{p/2-p\gamma}+
C\e\left[\tau_x^{1/2-\gamma}M(\tau_x)^{p-1}\right].
$$
Using \eqref{MC}, we conclude that the first two summands are bounded from above by $C(1+|x|^{p-2\gamma})$.
Applying the H\"older inequality with some $p'\in(p,p+p\gamma)$ to the third summand, we get 
$$
\e\left[\tau_x^{1/2-\gamma}M(\tau_x)^{p-1}\right]
\le\left(\e \tau_x^{p'/2-p'\gamma}\right)^{1/p'}\left(\e M^{(p-1)p'/(p'-1)}(\tau_x)\right)^{(p'-1)/p'}
$$
By \eqref{max-mom}, $\e M^\beta(\tau_x)\le C_\beta(1+|x|^\beta)$, $\beta<p$. 
From this inequality and from \eqref{MC}, we infer that
$$
\e\left[\tau_x^{1/2-\gamma}M(\tau_x)^{p-1}\right]
\le C(1+|x|^{1-\gamma})(1+|x|^{p-1}).
$$
As a result, \eqref{W1} holds also for $p\ge1$.
\end{proof}
\begin{lemma}
\label{lem:W2}
There exists $W^{(2)}(x)$ such that $W^{(2)}_k(x)$ converges a.s. and in $L^1$
towards $W^{(2)}(x)$. Moreover,
\begin{equation}
\label{W2}
\mathbf{E}[W^{(2)}(x)]\le C(1+|x|^{p-\gamma}).
\end{equation}
\end{lemma}
\begin{proof}
It is clear that all claims in the lemma will follow from
\begin{align}
\label{sum1.6}
\sum_{l=1}^\infty\e[|u(x+g_l+S(l))-u(x+g_{l-1}+S(l))|;\tau_x>l-1]
\le C(1+|x|^{p-\gamma}).
\end{align}

First we will consider  the case $p\ge 1$. 
Note that  if 
$x+g_{l-1}+S(l)\in K$ 
then, by \eqref{diff-bound}, 
$$
|u(x+g_l+S(l))-u(x+g_{l-1}+S(l))|
\le 
C|g_l-g_{l-1}|^p
+C |g_l-g_{l-1}||x+g_{l-1}+S(l)|^{p-1}
$$
and, similarly, 
if $x+g_{l}+S(l)\in K$ then 
$$
|u(x+g_l+S(l))-u(x+g_{l-1}+S(l))|
\le 
C|g_l-g_{l-1}|^p
+C |g_l-g_{l-1}||x+g_{l}+S(l)|^{p-1}.
$$
Hence, if 
either $x+g_l+S(l)\in K$ or $x+g_{l-1}+S(l)\in K$ then 
\begin{multline}
|u(x+g_l+S(l))-u(x+g_{l-1}+S(l))|\\ 
\le 
C|g_l-g_{l-1}|^p
+C |g_l-g_{l-1}|
(|x+g_{l-1}|^{p-1}
+|S(l-1)|^{p-1}+|X(l)|^{p-1}).
\label{eq.d1}
\end{multline}
Since $u=0$ outside of the cone the inequality \eqref{eq.d1} is obvious if both 
$x+g_l+S(l)\notin K$ and $x+g_{l-1}+S(l-1)\notin K$.
Using \eqref{eq.d1}, we have   
\begin{align}
\label{sum1.1}
\nonumber
 &
  \sum_{l=1}^k
  \left|
  \e[u(x+g_l+S(l))-u(x+g_{l-1}+S(l));\tau_x>l-1]\right|\\
  \nonumber 
  &\hspace{0.5cm}\le      
  C \sum_{l=1}^k \pr(\tau_x>l-1) \left(|g_l-g_{l-1}|^p+|g_l-g_{l-1}||g_l|^{p-1}\right)\\
  &\hspace{1cm}+
  C \sum_{l=1}^k
  |g_l-g_{l-1}|
  \e[|x|^{p-1}+|S(l-1)|^{p-1}+|X(l)|^{p-1};\tau_x>l-1].
 \end{align}  
By \eqref{MC}, 
noting that $|g_l-g_{l-1}| \le Cl^{-1/2-\gamma}$ we obtain 
for every $p\ge  1$,
\begin{align}
\label{sum1.2}
\nonumber
&\sum_{l=1}^k \pr(\tau_x>l-1) \left(|g_l-g_{l-1}|^p+|g_l-g_{l-1}||g_l|^{p-1}\right)\\
\nonumber
&\hspace{1cm}\le C\sum_{l=1}^k l^{p/2-1-p\gamma}\mathbf{P}(\tau_x>l-1)\\
&\hspace{1cm}\le C\mathbf{E}[\tau_x^{p/2-p\gamma}]\le C(1+|x|^{p-\gamma}). 
\end{align}
Similarly,
\begin{align}
\nonumber 
&(|x|^{p-1}+\e[|X(1)|]^{p-1})\sum_{l=1}^k|g_l-g_{l-1}|\pr(\tau_x>l)\\
\nonumber
&\hspace{2cm}\le 
C(1+|x|^{p-1}) \sum_{l=1}^k \mathbf{P}(\tau_x>l)l^{-1/2-\gamma}\\
&\hspace{2cm}\le C(1+|x|^{p-1})\mathbf{E}[\tau_x^{1/2-\gamma}]
\le C(1+|x|^{p-\gamma}).
\label{sum1.3}
\end{align}

It follows from \eqref{new.2a} and from Lemma~\ref{lem:tail} that
$$
\e\left[|S(l-1)|^{p-1};\tau_x>l-1\right]\le l^{(p-1)/2+\varepsilon (p-1)/2}\mathbf{P}(\tau_x>l-1)
+C\mathbf{E}[\tau_x]l^{-1/2-\varepsilon/2}
$$ 
in the case $p>2$. Therefore, for $\varepsilon<\gamma/(p-1)$,
\begin{align*}
&\sum_{l=1}^k|g_l-g_{l-1}|\e\left[|S(l-1)|^{p-1};\tau_x>l-1\right]\\
&\hspace{1cm}\le C\sum_{l=1}^k l^{p/2-1-\gamma/2}\mathbf{P}(\tau_x>l-1)
+C\mathbf{E}[\tau_x]\sum_{l=1}^k l^{-1-\gamma}\\
&\hspace{1cm}\le C\mathbf{E}[\tau_x^{p/2-\gamma/2}]+C\mathbf{E}[\tau_x].
\end{align*}
Then, taking into account \eqref{MC}, 
\begin{equation}
\label{sum1.5}
\sum_{l=1}^k|g_l-g_{l-1}|\e\left[|S(l-1)|^{p-1};\tau_x>l-1\right]\le C(1+|x|^{p-\gamma}).
\end{equation}
Similarly one shows that this relation is true in the case $p\le2$.
(Here one has to use the assumption $\mathbf{E}|X|^{2+\delta}<\infty$ instead of
$\mathbf{E}|X|^p<\infty$.)
Plugging \eqref{sum1.2}--\eqref{sum1.5} into \eqref{sum1.1}, we infer that
\eqref{sum1.6} holds in the case $p\ge1$.

Assume now that $p<1$. On the event $\{\tau_x>l-1\}$ one has 
$$
{\rm dist}(x+g_{l-1}+S(l-1),\partial K)\ge (l-1)^{1/2-\gamma}.
$$
This implies that 
$$
x+g_{l-1}+S(l)\in K\quad\text\quad
|x+g_{l-1}+S(l)|\ge\frac{(l-1)^{1/2-\gamma}}{2}
$$
provided that $\tau_x>l-1$ and $|X(l)|\le \frac{(l-1)^{1/2-\gamma}}{2}$.
From these observations and from \eqref{diff-bound1} we obtain
\begin{align*}
&\mathbf{E}[|u(x+g_{l}+S(l))-u(x+g_{l-1}+S(l))|;
\tau_x>l-1,|X(l)|\le (l-1)^{1/2-\gamma}/2]\\
&\hspace{1cm}\le
C|g_l-g_{l-1}|(l-1)^{(p-1)(1/2-\gamma)}\mathbf{P}(\tau_x>l-1).
\end{align*}
Furthermore, by \eqref{diff-bound2},
\begin{align*}
 &\mathbf{E}[|u(x+g_{l}+S(l))-u(x+g_{l-1}+S(l))|;
\tau_x>l-1,|X(l)|> (l-1)^{1/2-\gamma}/2]\\
&\le C|g_l-g_{l-1}|^p\mathbf{P}(\tau_x>l-1)
\mathbf{P}(|X(l)|>(l-1)^{1/2-\gamma}/2).
\end{align*}
Recalling that $\mathbf{E}|X(1)|^{2+\delta}<\infty$ and that
$|g_l-g_{l-1}|\le C(l-1)^{-1/2-\gamma}$ for $l>2$, one gets
easily
\begin{align*}
&\mathbf{E}[|u(x+g_{l}+S(l))-u(x+g_{l-1}+S(l))|;\tau_x>l-1]\\
&\hspace{1cm}\le C(l-1)^{p/2-1-p\gamma}\mathbf{P}(\tau_x>l-1),\ l\ge2.
\end{align*}
Summing over $l$ and using \eqref{MC} we complete the proof.
\end{proof}
\begin{lemma}
\label{lem:W3}
For every $x\in K$,
\begin{equation}
\label{W3_1}
\mathbf{E}\left[\sum_{l=0}^{\infty}|f(x+g_l+S(l))|\mathbb{I}\{\tau>l\}\right]
\le C\left(1+|x|^{p-\gamma}+\frac{|x|^p}{({\rm dist}(x,\partial K))^\gamma}\right)
\end{equation}
and
\begin{equation}
\label{W3_2}
\mathbf{E}[W^{(3)}_k(x)]\to
\mathbf{E}\left[\sum_{l=0}^{\infty}f(x+g_l+S(l))\mathbb{I}\{\tau>l\}\right].
\end{equation}
\end{lemma}
\begin{proof}
Recalling the definition of the function $f$, we have
$$
\mathbf{E}[W^{(3)}_k(x)]=
\e
\left[
\sum_{l=0}^{k-1}f(x+g_l+S(l))\mathbb{I}\{\tau>l\}
\right].
$$ 
This equality yields that \eqref{W3_2} is a simple consequence
of \eqref{W3_1}.

Applying Lemma~\ref{lem:bound.f} and using the elementary bound
${\rm dist}(y+g_l,\partial K)\ge l^{1/2-\gamma}$ for $y\in K$,  
following from the choice of $x_0$ and $R_0$ in the definition
of $g_l$, 
we have 
\begin{align*}
&\sum_{l=0}^{k-1} \e[|f(x+g_{l-1}+S(l-1))|;\tau_x>l-1]\\
&\hspace{1cm}\le C\frac{|x|^p}{({\rm dist}(x,\partial K))^{2+\delta}}
+C \sum_{l=1}^{k-1}
\e\left[\frac{|x+g_{l}+S(l)|^{p}}{{{\rm dist}(x+g_{l}+S(l),\partial K)^{2+\delta}}};\tau_x>l\right]\\
&\hspace{1cm}\le
C|x|^p\sum_{l=0}^{k-1} \mathbf{E}\left[({\rm dist}(x+g_{l}+S(l)),\partial K)^{-2-\delta};\tau_x>l\right]\\
&\hspace{2cm}+
C\sum_{l=1}^{k-1}|g_{l}|^{p-2-\delta}\mathbf{P}(\tau_x>l)\\
&\hspace{2cm}
+C\sum_{l=1}^{k-1}|g_{l}|^{-2-\delta}\mathbf{E}[|S(l)|^p;\tau_x>l].
\end{align*}  
Choosing $\gamma$ sufficiently small, we have
\begin{align}
\label{sum2.1}
\nonumber
\sum_{l=1}^{k-1}|g_{l}|^{p-2-\delta}\mathbf{P}(\tau_x>l)
&\le C\sum_{l=1}^{k-1} l^{p/2-1-\gamma}\mathbf{P}(\tau_x>l)\\
&\le C\mathbf{E}[\tau_x^{p/2-\gamma/2}]\le C(1+|x|^{p-\gamma}).
\end{align}
We next show that 
\begin{equation}
\label{sum2.2}
\sum_{l=1}^{k-1}|g_{l}|^{-2-\delta}\mathbf{E}[|S(l)|^p;\tau_x>l]
\le C(1+|x|^{p-\gamma}).
\end{equation}
Assume first that $p>2$.
Applying Lemma~\ref{lem:tail}, we have
\begin{align*}
\mathbf{E}[|S(l)|^p;\tau_x>l]\le
l^{p(1+\varepsilon)/2}\mathbf{P}(\tau_x>l)+C\mathbf{E}[\tau_x]. 
\end{align*}
Combining this with the estimate $|g_l|\ge cl^{1/2-\gamma}$, we obtain,
for $\gamma<\frac{\delta}{5+2\delta}$ and 
$\varepsilon<\frac{\delta-5\gamma-2\gamma\delta}{p}$,
\begin{align*}
&\sum_{l=1}^{k-1}|g_{l}|^{-2-\delta}\mathbf{E}[|S(l)|^p;\tau_x>l]\\
&\hspace{1cm}\le C\sum_{l=1}^{k-1}
l^{p(1+\varepsilon)/2-(2+\delta)(1/2-\gamma)}\mathbf{P}(\tau_x>l)
+C\mathbf{E}[\tau_x]\sum_{l=1}^k l^{-(2+\delta)(1/2-\gamma)}\\
&\hspace{1cm}\le C\sum_{l=1}^{k-1}
l^{p/2-\gamma/2-1}\mathbf{P}(\tau_x>l)
+C\mathbf{E}[\tau_x]\\
&\hspace{1cm}\le C \mathbf{E}[\tau_x^{p/2-\gamma/2}].
\end{align*}
Taking into account Lemma~\ref{moments}, we get \eqref{sum2.2} for $p>2$.

If $p\le2$ then the moment of order $2+\delta$ is finite and, consequently,
\begin{align*}
&\mathbf{E}[|S(l)|^p;\tau_x>l]\\
&\hspace{1cm}\le l^{p(1+\varepsilon)/2}\mathbf{P}(\tau_x>l)
+l^{-(2+\delta-p)(1+\varepsilon/2)}
\mathbf{E}[|S(l)|^{2+\delta};|S(l)|>l^{(1+\varepsilon)/2},\tau_x>l]\\
&\hspace{1cm}\le l^{p(1+\varepsilon)/2}\mathbf{P}(\tau_x>l)
+Cl^{-(2+\delta-p)(1+\varepsilon/2)}\mathbf{E}[\tau_x\wedge l],
\end{align*}
where in the last step we used Lemma~\ref{lem:tail} once again. Noting that
$$
\mathbf{E}[\tau_x\wedge l]\le 
l^{1-p/2+\gamma/2}\mathbf{E}[\tau_x^{p/2-\gamma/2}]
$$
we obtain
\begin{align*}
&\mathbf{E}[|S(l)|^p;\tau_x>l]\\
&\hspace{1cm}\le l^{p(1+\varepsilon)/2}\mathbf{P}(\tau_x>l)
+Cl^{-(2+\delta-p)(1+\varepsilon/2)+1-p/2+\gamma/2}
\mathbf{E}[\tau_x^{p/2-\gamma/2}].
\end{align*}
Summing over $l$ we infer that \eqref{sum2.2} holds also for $p\le2$,
provided that $\gamma$ and $\varepsilon$ are sufficiently small.

Using the bound
${\rm dist}(y+g_l,\partial K)\ge l^{1/2-\gamma}$ and choosing $\gamma$ sufficiently small,
we conclude that
\begin{align*}
&\sum_{l=g(x)}^\infty \mathbf{E}\left[({\rm dist}(x+g_{l}+S(l)),\partial K)^{-2-\delta};\tau_x>l\right]\\
&\hspace{1cm}\le  \sum_{l=g(x)}^\infty l^{-(1/2-\gamma)(2+\delta)}
\le Cg^{-\gamma}(x).
\end{align*}
Furthermore, if $|S(l)|\le g(x)/2$ then ${\rm dist}(x+g_l+S(l),\partial K)>g(x)/2$. Thus,
by the Chebyshev inequality,
$$
\mathbf{E}\left[({\rm dist}(x+g_{l}+S(l)),\partial K)^{-2-\delta};\tau_x>l\right]
\le C g(x)^{-2-\delta}+Cl^{-(1/2-\gamma)(2+\delta)}\frac{l}{g^2(x)}.
$$
Consequently,
\begin{align*}
\sum_{l=0}^{g(x)}\mathbf{E}\left[({\rm dist}(x+g_{l}+S(l)),\partial K)^{-2-\delta};\tau_x>l\right]
\le Cg^{-\gamma}(x).
\end{align*}
As a result, 
$$
\sum_{l=0}^\infty \mathbf{E}\left[({\rm dist}(x+g_{l}+S(l)),\partial K)^{-2-\delta};\tau_x>l\right]\le Cg^{-\gamma}(x).
$$
Combining this with \eqref{sum2.1} and \eqref{sum2.2},
we arrive at \eqref{W3_1}.
\end{proof}

The claim of Proposition~\ref{prop:main} is immediate from Lemmas~\ref{lem:W1},
\ref{lem:W2} and \ref{lem:W3}. Furthermore, we have, for a sufficiently small
$\gamma$, the estimate
\begin{equation}
\label{V-est}
|V(x)-u(x)|\le C\left(1+|x|^{p-\gamma}
+\frac{|x|^p}{({\rm dist}(x,\partial K))^\gamma}\right).
\end{equation}

 \begin{lemma}
\label{lem:positivity}
  The function $V$ possesses the following properties.
  \begin{itemize}
  \item[(a)] For any $\gamma>0, R>0$, uniformly in $x\in D_{R,\gamma}$ we have $V(tx)\sim u(tx)$ as $t\to\infty$.
  \item[(b)] For all $x\in K$ we have $V(x)\leq C(1+|x|^p)$.
  \item[(c)] The function $V$ is harmonic for the killed random walk, that is
    $$
    V(x)=\mathbf{E}\left[V(x+S(n_0)),\tau_x>n_0\right],\quad x\in K, n_0\geq1.
    $$
  \item[(d)] The function $V$ is strictly positive on $K_+$.
  \item[(e)] If $x\in K$, then $V(x)\leq V(x+x_0)$, for all  $x_0$  such that $x_0+K\subset K$.
  \end{itemize}
\end{lemma}
The proof is identical with that of Lemma 13 in \cite{DW15}, for the proof of (c) one has to notice that 
\eqref{V-est} implies that $V(x)=u(x)+O(|x|^{p-\gamma})$ for $x\in D_{R,\gamma}$.
\section{Proof of Corollary~\ref{cor:int.lim}}
As we have mentioned in the introduction, 
the proofs of the claims in the corollaries 
are quite close to that in \cite{DW15}. 
We demonstrate the needed changes by deriving 
the tail asymptotics for $\tau_x$. 
Proofs of other results can similarly be adapted to the present 
setting.

Let 
$$
K_{n,\varepsilon}:=\left\{x\in K\,:\,{\rm dist}(x,\partial K)\ge n^{1/2-\varepsilon}\right\},
$$
where $\varepsilon<\gamma/(1+p).$
Similarly to  the proof of Lemma 20 in \cite{DW15} 
one can show that 
\begin{equation}\label{eq.lem20}
u(y+g_k)=(1+o(1)) u(y), \quad 
y\in K_{n,\varepsilon}, 
|y|\le\sqrt{n}
\end{equation}
uniformly in $k\le n^{1-\varepsilon}$. 
Indeed, by \eqref{diff-bound1}, 
$$
|u(y+g_k)-u(y)| 
\le C |g_k| |y|^{p-1}
=C\frac{|g_k|}{|y|} |y|^{p}
\le C n^{\varepsilon-\gamma} n^{p/2}.
$$
Also, by the lower bound in \eqref{u-bounds} 
$$
|u(y)| \ge C n^{p/2-p\varepsilon}. 
$$
Hence, 
$$
\frac{|u(y+g_k)-u(y)|}{|u(y)|} \le Cn^{(1+p)\varepsilon-\gamma},
$$
which proves \eqref{eq.lem20}, as $(1+p)\varepsilon<\gamma$. 

Then, it follows from \eqref{eq.lem20} that for any sequence $\theta_n\to0$,
\begin{align*}
&\mathbf{E}\left[u(x+S(\nu_n));\tau_x>\nu_n,\nu_n\le n^{1-\varepsilon},
|x+S(\nu_n)|\le\theta_n\sqrt{n}\right]\\
&\hspace{1cm}\sim
\mathbf{E}\left[u(x+g_{\nu_n}+S(\nu_n));\tau_x>\nu_n,
\nu_n\le n^{1-\varepsilon},|x+S(\nu_n)|\le\theta_n\sqrt{n}\right],
\end{align*}
where
$$
\nu_n:=\inf\{n\ge0:x+S(n)\in K_{n,\varepsilon}\}.
$$
Applying this to (50) in \cite{DW15}, we obtain
\begin{align*}
&\mathbf{P}(\tau_x>n,\nu_n\le n^{1-\varepsilon})\\
&\hspace{0.5cm}=\frac{\varkappa+o(1)}{n^{p/2}}\mathbf{E}\left[u(x+g_{\nu_n}+S(\nu_n));
\tau_x>\nu_n, \nu_n\le n^{1-\varepsilon}\right]\\
&\hspace{1cm}+
O\left(\frac{1}{n^{p/2}}\mathbf{E}\left[|x+S(\nu_n)|^p;
\tau_x>\nu_n, \nu_n\le n^{1-\varepsilon},|x+S(\nu_n)|>\theta_n\sqrt{n}\right]\right).
\end{align*}
In view of Lemma 24 in \cite{DW15}, $O$-term is $o(n^{-p/2})$. Therefore, the 
proof will be completed if we show that
\begin{equation}
\label{final-step}
\mathbf{E}\left[u(x+g_{\nu_n}+S(\nu_n));\tau_x>\nu_n, \nu_n\le n^{1-\varepsilon}\right]\to V(x).
\end{equation}
According to \eqref{decomp_main},
\begin{align*}
&\mathbf{E}\left[u(x+g_{\nu_n}+S(\nu_n));\tau_x>\nu_n, \nu_n\le n^{1-\varepsilon}\right]\\
&\hspace{1cm}=\mathbf{E}[u(x)-W^{(1)}_{\nu_n}(x)+W^{(2)}_{\nu_n}(x)+W^{(3)}_{\nu_n}(x);\nu_n\le n^{1-\varepsilon}].
\end{align*}
Now note that $\mathbf{P}(\nu_n\le n^{1-\varepsilon})\to 1$ 
by the functional central limit theorem. Therefore, 
since $\nu_n\to\infty$, as $n\to\infty$, 
by  Lemmas \ref{lem:W1} and \ref{lem:W2} 
and the dominated convergence theorem we obtain, as $n\to\infty$,
\begin{align*}
&\mathbf{E}[u(x)-W^{(1)}_{\nu_n}(x)+W^{(2)}_{\nu_n}(x);
\nu_n\le n^{1-\varepsilon}]\\
&\hspace{2cm}\to
u(x)-\mathbf{E}u(x+g_{\tau_x}+S(\tau_x))+\mathbf{E}[W^{(2)}(x)].
\end{align*}
Thus, \eqref{final-step} will follow from the convergence
\begin{equation}
\label{final-step1} 
\mathbf{E}[W^{(3)}_{\nu_n}(x);\nu_n\le n^{1-\varepsilon}]
\to\mathbf{E}\left[\sum_{l=0}^{\infty}f(x+g_l+S(l))\mathbb{I}\{\tau>l\}\right].
\end{equation}
It is immediate from the definition of $f(x)$ that the sequence
$$
Y_k:=W^{(3)}_k(x)-\sum_{l=0}^{k-1}f(x+g_l+S(l))\mathbb{I}\{\tau_x>l\}
$$
is a martingale. Upper bound \eqref{W3_1} 
together with the dominated convergence theorem 
imply that 
\begin{align}
  \label{final-step2}
  \nonumber
  &\mathbf{E}\left[\sum_{l=0}^{\nu_n-1}f(x+g_l+S(l))\mathbb{I}\{\tau_x>l\};
  \nu_n\le n^{1-\varepsilon}\right]\\
  &\hspace{2cm}\to
  \mathbf{E}\left[\sum_{l=0}^{\infty}f(x+g_l+S(l))\mathbb{I}\{\tau_x>l\}\right].
  \end{align}
Furthermore, by the optional stopping theorem,
\begin{align*}
&\mathbf{E}[Y_{\nu_n};\nu_n\le n^{1-\varepsilon}]
=
\mathbf{E}[Y_{\nu_n\wedge n^{1-\varepsilon}};\nu_n\le n^{1-\varepsilon}]
=\mathbf{E}[Y_{\nu_n\wedge n^{1-\varepsilon}}]-
\mathbf{E}[Y_{\nu_n\wedge n^{1-\varepsilon}};\nu_n> n^{1-\varepsilon}]
\\
&=-\mathbf{E}[Y_{n^{1-\varepsilon}};\nu_n> n^{1-\varepsilon}]\\
&=-\mathbf{E}[W^{(3)}_{n^{1-\varepsilon}}(x);\nu_n>n^{1-\varepsilon}]
+\mathbf{E}\left[\sum_{l=0}^{n^{1-\varepsilon}-1}f(x+g_l+S(l))
\mathbb{I}\{\tau_x>l\};\nu_n>n^{1-\varepsilon}\right].
\end{align*}
Using \eqref{W3_1}, the dominated convergence theorem and the fact that $\mathbf P(\nu_n>n^{1-\varepsilon})$
once again, we infer that
$$
\mathbf{E}[Y_{\nu_n};\nu_n\le n^{1-\varepsilon}]
=-\mathbf{E}[W^{(3)}_{n^{1-\varepsilon}}(x);\nu_n>n^{1-\varepsilon}]+o(1).
$$
Assume first that $p<1$. Recalling the definition of $W^{(3)}_k$ and
using \eqref{diff-bound2}, we have
$$
|W^{(3)}_k|\le C\sum_{l=1}^{k}|X(l)|^p\mathbb{I}\{\tau_x>l-1\}.
$$
Then, for every fixed $N\ge1$,
\begin{align*}
&\mathbf{E}[|W^{(3)}_{n^{1-\varepsilon}}(x)|;\nu_n>n^{1-\varepsilon}]\\
&\hspace{0.5cm}\le
C\mathbf{E}\left[\sum_{l=1}^{N}|X(l)|^p;\nu_n>n^{1-\varepsilon}\right]
+C\sum_{l=N+1}^{n^{1-\varepsilon}}\mathbf{E}|X(1)|^p
\mathbf{P}(\tau_x>l-1,\nu_n>l-1).
\end{align*}
The first summand on the right hand side converges to zero due to the fact
that $\mathbf{P}(\nu_n>n^{1-\varepsilon})\to0$. Furthermore, by Lemma 14 in \cite{DW15}, 
\begin{equation}
\label{old_lemma_14}
\mathbf{P}(\nu_n>n^{1-\varepsilon},\tau_x>n^{1-\varepsilon})
\le\exp\{-Cn^\varepsilon\}.
\end{equation}

Therefore, we obtain
we get
\begin{align*}
\sum_{l=N+1}^{n^{1-\varepsilon}}\mathbf{P}(\tau_x>l-1,\nu_n>l-1)
&\le \sum_{l=N+1}^{n^{1-\varepsilon}}\mathbf{P}(\tau_x>l-1,\nu_l>l-1)\\
&\le \sum_{l=N+1}^{\infty} e^{-Cl^\varepsilon}.
\end{align*}
Letting here $N\to\infty$, we conclude that
$$
\mathbf{E}[|W^{(3)}_{n^{1-\varepsilon}}(x)|;\nu_n>n^{1-\varepsilon}]\to0.
$$
It remains to prove this relation for $p\ge1$. In this case, using
\eqref{diff-bound}, we have
$$
|W^{(3)}_k|\le C\sum_{l=1}^{k}\left(|X(l)|^p+|X(l)||x+g_{l-1}+S(l-1)|^{p-1}\right)\mathbb{I}\{\tau_x>l-1\}.
$$
The summands with $|X(l)|^p$ have been already considered. For the remaining summands we have 
\begin{align*}
&\mathbf{E}\left[\sum_{l=1}^{n^{1-\varepsilon}}
|X(l)||x+g_{l-1}+S(l-1)|^{p-1}\mathbb{I}\{\tau_x>l-1\};\nu_n>n^{1-\varepsilon}\right]\\
&\le\mathbf{E}\left[\sum_{l=1}^{N}
|X(l)||x+g_{l-1}+S(l-1)|^{p-1};\nu_n>n^{1-\varepsilon}\right]\\
&\hspace{1cm}+\mathbf{E}|X(1)|\sum_{l=N+1}^{n^{1-\varepsilon}}
\mathbf{E}\left[|x+g_{l-1}+S(l-1)|^{p-1};\tau_x>l-1,\nu_n>l-1\right].
\end{align*}
The first summand converges again to zero
since $\mathbf{P}(\nu_n>n^{1-\varepsilon})\to0$. 
For the second summand 
applying the Cauchy-Schwarz inequality and then
using \eqref{old_lemma_14}, we have
\begin{align*}
&\mathbf{E}\left[|x+g_{l-1}+S(l-1)|^{p-1};\tau_x>l-1,\nu_n>l-1\right]\\
&\hspace{1cm}\le \left(\mathbf{E}\left[|x+g_{l-1}+S(l-1)|^{p}\right]\right)^{(p-1)/p}
\mathbf{P}^{1/p}(\tau_x>l-1,\nu_n>l-1)\\
&\hspace{1cm}\le Cl^{p/2}e^{-Cl^\varepsilon}.
\end{align*}
Therefore, letting $N\to\infty$, we complete the proof.
\section{Second proof of Theorem~\ref{thm:C2}}

For every $\varepsilon>0$ define
$$
\widetilde{K}_{n,\varepsilon}:= 
\left\{x\in K\,:\,{\rm dist}(x,\partial K)\ge \frac{1}{2}
\left(n^{1/2-\varepsilon}+\frac{|x|}{n^{2\varepsilon}}\right)\right\}.
$$

\subsection{Preliminary estimates.}
The next statement is the most important step in this proof of Theorem~\ref{thm:C2}.
\begin{proposition}
\label{prop:iteration}
Assume that the conditions of Theorem~\ref{thm:C2} are valid. 
Then, for every sufficiently small $\varepsilon>0$ there exists
$q>0$ such that
$$
\max_{k\le n}\Big|\mathbf{E}[u(x+S(k));\tau_x>k]-u(x)\Big|\le\frac{C}{n^q}u(x),
\quad x\in \widetilde{K}_{n,\varepsilon}.
$$
\end{proposition}
\begin{lemma}\label{lem:A}
  For sufficiently small $\varepsilon$ there exists $q>0$ such that 
  $$
  \max_{k\in [\sqrt n,n]}\left|
  \e[u(x+S(k)),\tau_x>k]-u(x)
  \right|\le \frac{C}{n^q},\quad x\in\widetilde K_{n,\varepsilon}.
  $$
\end{lemma}  
\begin{proof}
For every $x\in K$ define
$$
x^{+}_k=x+g_k=x+ k^{1/2-\gamma}R_0x_0.
$$
Clearly,
\begin{align*}
&\mathbf{E}[u(x+S(k));\tau_x>k]\\
&\hspace{0.5cm}=\mathbf{E}[u(x+S(k))-u(x_k^++S(k));\tau_x>k]+\e[u(x_k^++S(k));\tau_x>k].
\end{align*}
If $p\ge1$ then, using \eqref{diff-bound}, we get
\begin{align*}
&\mathbf{E}[|u(x+S(k))-u(x_k^++S(k))|;\tau_x>k]\\
&\hspace{1cm}\le Ck^{1/2-\gamma}\mathbf{E}\left[|x|^{p-1}+|S(k)|^{p-1}+k^{(1/2-\gamma)(p-1)}\right]\\
&\hspace{1cm} \le Ck^{1/2-\gamma}\left(|x|^{p-1}+k^{(p-1)/2}\right).
\end{align*}
In the case $p<1$ we use \eqref{diff-bound2} to obtain
\begin{align*}
\mathbf{E}[|u(x+S(k))-u(x_k^++S(k))|;\tau_x>k]\le C|g_k|^p\le Ck^{p(1/2-\gamma)}.
\end{align*}
Combining these two cases, we have
\begin{align*}
\mathbf{E}[|u(x+S(k))-u(x_k^++S(k))|;\tau_x>k]
\le Ck^{1/2-\gamma}\left(|x|^{p-1}{\bf 1}\{p\ge1\}+k^{(p-1)/2}\right).
\end{align*}
Next, 
\begin{align*}
&\e[u(x_k^++S(k));\tau_x>k]\\
&\hspace{1cm}=
u(x_k^+)+
\sum_{l=1}^k
\left(
  \e[u(x_k^++S(l));\tau_x>l] - \e[u(x_k^++S(l-1));\tau_x>l-1]
  \right)\\
&\hspace{1cm}=
u(x_k^+)+
\sum_{l=1}^k
\e[u(x_k^++S(l))-u(x_k^++S(l-1);\tau_x>l-1]\\ 
&\hspace{2cm} - \sum_{l=1}^k \e[u(x_k^++S(l));\tau_x=l]\\
&\hspace{1cm}=
u(x_k^+)+
\sum_{l=1}^k 
\e[f(x_k^++S(l-1));\tau_x>l-1] 
-\e[u(x_k^++S(\tau_x));\tau_x\le k].
\end{align*}
Using \eqref{diff-bound} and \eqref{diff-bound2} once again, we have
\begin{equation}\label{eq:first}
|u(x_k^+) -u(x)|\le Ck^{1/2-\gamma}\left(|x|^{p-1}{\bf 1}\{p\ge1\}+k^{(1/2-\gamma)(p-1)}\right).
\end{equation}
By Lemma \ref{lem:bound.f},
\begin{align*}
 \left| \sum_{l=1}^k 
  \e[f(x_k^++S(l-1));\tau_x>l-1] \right|
  &\le C 
  \sum_{l=0}^{k-1} \e\left[\frac{|x_k^++S(l)|^p}{\mbox{dist}(x_k^++S(l),\partial K)^{2+\delta}};\tau_x>l\right].
\end{align*}  
Now note that on the event $\{\tau_x>l\}$ the random variable $x+S(l)\in K$. Hence 
$\mbox{dist}(x_k^++S(l),\partial K)\ge C k^{1/2-\gamma}$. 
Therefore, 
\begin{align}
  \left|\sum_{l=1}^k 
  \e[f(x_k^++S(l-1);\tau_x>l-1]\right| 
  &\le C 
  \sum_{l=0}^{k-1} \e\left[\frac{|x_k^++S(l)|^p}{k^{(1/2-\gamma)(2+\delta)}};\tau_x>l\right]\nonumber\\
  \le C\frac{k|x_k^+|^p+kk^{p/2}}{k^{(1/2-\gamma)(2+\delta)}}
  &\le C \frac{
               |x|^p+k^{p(1/2-\gamma)}+k^{p/2}}
               {k^{(1/2-\gamma)(2+\delta)-1}} \nonumber\\
  &\le C \frac{|x|^p+k^{p/2}}{k^{(1/2-\gamma)(2+\delta)-1}}. 
  \label{eq:second}
\end{align}  
If $p\ge1$ then, using \eqref{u-bounds} and the fact that $u(x)=0$ for $x\notin K$, we obtain 
\begin{align*}
\e[u(x_k^++S(\tau_x));\tau_x\le k]
&\le \e[|x_k^++S(\tau_x)|^{p-1}\mbox{dist}(x_k^++S(\tau_x),\partial K),\tau_x\le k]\\
&\le Ck^{1/2-\gamma} \e[|x_k^++S(\tau_x)|^{p-1};\tau_x\le k]\\
&\le Ck^{p(1/2-\gamma)}+Ck^{1/2-\gamma} \e[|x+S(\tau_x)|^{p-1};\tau_x\le k].
\end{align*}
To bound the second term we use the Burkholder inequality, 
\begin{align*}
\e[|x+S(\tau_x)|^{p-1};\tau_x\le k]
&\le C|x|^{p-1} +C \e[\max_{l\le k } |S(l)|^{p-1}]\\
&\le C|x|^{p-1} +C k^{(p-1)/2}. 
\end{align*}
Then, 
\begin{equation}
  \label{eq:third}
  \e[u(x_k^++S(\tau_x));\tau_x\le k] 
  \le Ck^{p(1/2-\gamma)}+ Ck^{p/2-\gamma}
+Ck^{1/2-\gamma} |x|^{p-1}.
\end{equation}  
If $p<1$ then, applying \eqref{diff-bound2}, we obtain
$$
\e[u(x_k^++S(\tau_x));\tau_x\le k] \le Ck^{p/2-p\gamma}.
$$
In other words, \eqref{eq:third} holds also for $p<1$.

Combining now \eqref{eq:first}, \eqref{eq:second} and \eqref{eq:third},
we obtain 
\begin{align*}
&|\e[u(x+S(k));\tau_x>k]-u(x)|\\
&\hspace{1cm}\le 
C\left(k^{1/2-\gamma}|x|^{p-1}+k^{p(1/2-\gamma)}
+k^{p/2-\gamma}+\frac{|x|^p+k^{p/2}}{k^{(1/2-\gamma)(2+\delta)-1}}
\right).
\end{align*}
We can assume that $\gamma<1/2 $ is sufficiently small to ensure 
that 
$\gamma<p/2$ and 
$$
p/2>(1/2-\gamma)(2+\delta)-1
=\delta/2-2\gamma-\gamma\delta
>0. 
$$
Then, 
\begin{align*}
  &\max_{\sqrt n\le k\le n }|\e[u(x+S(k));\tau_x>k]-u(x)|\\
  &\hspace{0.5cm}\le 
  C\biggl(
    n^{1/2-\gamma}|x|^{p-1}+n^{p(1/2-\gamma)}
    +n^{p/2-\gamma}
    +n^{p/2-(\delta/2-2\gamma-\gamma\delta)}
+\frac{|x|^p}{n^{\delta/4-\gamma-\gamma\delta/2}}
  \biggr).
\end{align*}  

For every $x\in\widetilde{K}_{n,\varepsilon}$ one has
$$
|x|\le 2n^{2\varepsilon}{\rm dist}(x,\partial K)
\quad\text{and}\quad
{\rm dist}(x,\partial K)\ge \frac{1}{2}n^{1/2-\varepsilon}.
$$
Combining these estimates with the lower bound in \eqref{u-bounds}, we obtain
\begin{equation}
  \label{Ac.2} 
  |x|^{p}\le\frac{2^{p}n^{2p\varepsilon}}{C_1}u(x),
\end{equation}
\begin{equation}
\label{Ac.3}
|x|^{p-1}\le\frac{2^{p-1}n^{2(p-1)\varepsilon}}{C_1}\frac{u(x)}{{\rm dist}(x,\partial K)}
\le \frac{2^p}{C_1}n^{(2p-1)\varepsilon}\frac{u(x)}{n^{1/2}}
\end{equation}
and
\begin{equation}
\label{Ac.4}
n^{p/2}\le 2^pn^{p\varepsilon}\left({\rm dist}(x,\partial K)\right)^p\le\frac{2^p}{C_1}u(x)n^{p\varepsilon}.
\end{equation}
Taking into account \eqref{Ac.2},\eqref{Ac.3} and \eqref{Ac.4}, we arrive at the bound
\begin{align*}
  &\max_{\sqrt n\le k\le n }|\e[u(x+S(k));\tau_x>k]-u(x)|\\
  &\le 
  Cu(x)\biggl(
    n^{(2p-1)\varepsilon-\gamma}+n^{p(\varepsilon-\gamma)}
    +n^{p\varepsilon-\gamma}
    +n^{p\varepsilon-(\delta/2-2\gamma-\gamma\delta)}
    +n^{2p\varepsilon-(\delta/4-\gamma-\gamma\delta/2)}
  \biggr).
\end{align*}   
Clearly, we can pick sufficiently small $\varepsilon>0$ in such a way 
that all exponents on the right hand side of the previous inequality are negative.
This completes the proof of the lemma.
\end{proof}
\begin{proof}[Proof of Proposition~\ref{prop:iteration}]
If $k\in[\sqrt{n},n]$ then the desired estimate is immediate from Lemma~\ref{lem:A}. Thus, it remains to consider the case $k<\sqrt{n}$. Clearly, 
\begin{align}
\label{prop.proof.1}
\nonumber
&\mathbf{E}[u(x+S(k));\tau_x>k]-u(x)\\
&\hspace{1cm}=\mathbf{E}[u(x+S(k))-u(x);\tau_x>k]-u(x)\mathbf{P}(\tau_x\le k)
\end{align}
By the Doob inequality, for every $x\in\widetilde{K}_{n,\varepsilon}$,
\begin{align}
\label{prop.proof.2}
\mathbf{P}(\tau_x\le k)\le\mathbf{P}\left(\max_{j\le k}|S(j)|^2\ge n^{1-2\varepsilon}\right)
\le C\frac{k}{n^{1-2\varepsilon}}.
\end{align}
Using \eqref{diff-bound} and \eqref{diff-bound2}, we conclude that, for all $k\le\sqrt{n}$,
\begin{align*}
\mathbf{E}\left[|u(x+S(k))-u(x)|;\tau_x>n\right]
&\le C\mathbf{E}\left[|S(k)||x|^{p-1}{\bf 1}\{p\ge1\}+|S(k)|^p\right]\\
&\le C\left(n^{1/4}|x|^{p-1}+n^{p/4}\right).
\end{align*}
Taking into account \eqref{Ac.3} and \eqref{Ac.4}, we obtain
\begin{equation}
\label{prop.proof.3}
\max_{k\le\sqrt{n}}\mathbf{E}\left[|u(x+S(k))-u(x)|;\tau_x>n\right]
\le\frac{C}{n^q}u(x),\quad x\in\widetilde{K}_{n,\varepsilon}.
\end{equation}
Combining \eqref{prop.proof.1}--\eqref{prop.proof.3} completes the proof of the proposition.
\end{proof}
Define
$$
\nu_n:=\inf\left\{n\ge0:x+S(n)\in K_{n,\varepsilon}\right\}
$$
and
$$
\widetilde{\nu}_n:=\inf\left\{n\ge0:x+S(n)\in \widetilde{K}_{n,\varepsilon}\right\}.
$$
\begin{lemma}
\label{lem:nu}
There exists $\gamma>0$ such that, for every $x\in K$,
$$
\max_{k\in[n^{1-\varepsilon},n]}
\mathbf{E}\left[u(x+S(k));\tau_x>k,\widetilde{\nu}_n> [n^{1-\varepsilon}]\right]\le \frac{C(1+|x|^{p-\gamma})}{n^q}.
$$
\end{lemma}
\begin{proof}
Set 
$$
M(k):=\max_{j\le k}|x+S(j)|
$$
and split the expectation into two parts:
\begin{align}
\label{nu.1}
\nonumber
&\mathbf{E}\left[u(x+S(k));\tau_x>k,\widetilde{\nu}_n>[n^{1-\varepsilon}]\right]\\
\nonumber
&\hspace{1cm}=\mathbf{E}\left[u(x+S(k));\tau_x>k,\widetilde{\nu}_n>[n^{1-\varepsilon}],M([n^{1-\varepsilon}])\le n^{1/2+\varepsilon/2}\right]\\
&\hspace{2cm}+\mathbf{E}\left[u(x+S(k));\tau_x>k,\widetilde{\nu}_n>[n^{1-\varepsilon}],M([n^{1-\varepsilon}])> n^{1/2+\varepsilon/2}\right].
\end{align}
Set, for brevity, $m=[n^{1-\varepsilon}]$. Now, using the relation, 
\begin{align*}
  &\left\{\widetilde{\nu}_n> m, M(m)\le n^{1/2+\varepsilon/2}\right\}\\
  &\subset\left\{{\rm dist}(x+S(j),\partial K)\le \frac{1}{2}\left(n^{1/2-\varepsilon}+\frac{|x+S(j)|}{n^{2\varepsilon}}\right),
  |x+S(j)|\le n^{1/2+\varepsilon/2}, j\le m\right\}\\
  &\subset\left\{{\rm dist}(x+S(j),\partial K)\le n^{1/2-\varepsilon}, j\le m\right\}
  =\left\{\nu_n> m\right\},
  \end{align*}
we obtain 
\begin{align*}
  &\mathbf{E}\left[u(x+S(k));\tau_x>k,\widetilde{\nu}_n> n^{1-\varepsilon},M(m)\le n^{1/2+\varepsilon/2}\right]\\
  &\hspace{1cm}\le \mathbf{E}\left[u(x+S(k));\tau_x>k,\nu_n> m,M(m)\le n^{1/2+\varepsilon/2}\right].
  \end{align*}
  
Now, if $p\ge 1$ then by \eqref{diff-bound}, 
\begin{align*}
  u(x+S(k))\le u(x+S(m))
  +C|S(k)-S(m)|^{p}
  +C|S(k)-S(m)|
  |x+S(m)|^{p-1}.  
\end{align*}  
If $p<1$ we make use of \eqref{diff-bound2} to obtain 
\begin{align*}
    u(x+S(k))\le u(x+S(m))
    +C|S(k)-S(m)|^{p}
  \end{align*}  
As a result,
\begin{align}
\label{nu.1+}
\nonumber
u(x+S(k))\le u(x+S(m))
  &+C|S(k)-S(m)|^{p}\\
  &+C|S(k)-S(m)|
  |x+S(m)|^{p-1}{\bf 1}\{p\ge1\}. 
\end{align}

Hence, 
\begin{align*}
  &\mathbf{E}\left[u(x+S(k));\tau_x>k,\nu_n> m,M(m)\le n^{1/2+\varepsilon/2}\right]
 \\
 &\le 
 \mathbf{E}\left[u(x+S(m));\tau_x>k,\nu_n> m,M(m)\le n^{1/2+\varepsilon/2}\right]\\
 &+
 C\mathbf{E}\left[|S(k)-S(m)|^p;\tau_x>k,\nu_n> m,M(m)\le n^{1/2+\varepsilon/2}\right]\\
 &+
 C{\bf 1}\{p\ge1\}\mathbf{E}\left[|S(k)-S(m)||x+S(m)|^{p-1};\tau_x>k,\nu_n> m,M(m)\le n^{1/2+\varepsilon/2}\right].
\end{align*}  
First, since $u(y)\le C|y|^p$,
by \eqref{old_lemma_14}, we conclude that
\begin{multline}
\label{nu.2.1}
\max_{k\in[n^{1-\varepsilon},n]}
\mathbf{E}\left[u(x+S(m));\tau_x>k,\nu_n> m,M(m)\le n^{1/2+\varepsilon/2}\right]
\\ 
\le Cn^{p/2+\varepsilon/2}
\pr\left(\tau_x>m,\nu_n> m\right)
\le Cn^{p/2+\varepsilon/2} e^{-cn^\varepsilon}
\end{multline}
and
\begin{align}
  \label{nu.2.2}
  \nonumber
  &\max_{k\in[n^{1-\varepsilon,n}]}
  \mathbf{E}\left[|S(k)-S(m)|^p;\tau_x>k,\nu_n> m,M(m)\le n^{1/2+\varepsilon/2}\right]\\
  \nonumber
  &\hspace{1cm}\le 
  \max_{k\in[n^{1-\varepsilon,n}]}\mathbf{E}\left[|S(k)-S(m)|^p\right]\pr\left(\tau_x>m,\nu_n> m\right)\\
  &\hspace{1cm}\le C n^{p/2}e^{-cn^\varepsilon}. 
\end{align}  
Second, for $p\ge 1$ one has by the same argument,  
\begin{align}
  \label{nu.2.3}
  \nonumber
  &\max_{k\in[n^{1-\varepsilon,n}]}
  \e\left[|S(k)-S(m)||x+S(m)|^{p-1};\tau_x>k,\nu_n> m,M(m)\le n^{1/2+\varepsilon/2}\right] \\
  &\hspace{1cm}\le Cn^{1/2}n^{(p-1)/2+\varepsilon/2}e^{-cn^\varepsilon}.
\end{align}
Therefore, combining  \eqref{nu.2.1}, \eqref{nu.2.2} and \eqref{nu.2.3}
we obtain, 
that the first expectation on the right hand side of \eqref{nu.1}
can be estimated as follows, 
\begin{align}
  \label{nu.2}
  \nonumber
  &\max_{k\in[n^{1-\varepsilon,n}]} \mathbf{E}\left[u(x+S(k));\tau_x>k,\widetilde{\nu}_n>[n^{1-\varepsilon}],M([n^{1-\varepsilon}])\le n^{1/2+\varepsilon/2}\right]\\
  &\hspace{1cm}\le 
  Cn^{p/2+\varepsilon/2} e^{-cn^\varepsilon}. 
\end{align}  
Using \eqref{nu.1+}, we have for the second expectation on the right hand side of \eqref{nu.1},
\begin{align*}
&
\mathbf{E}\left[u(x+S(k));\tau_x>k,
\widetilde{\nu}_n> m,M(m)>  n^{1/2+\varepsilon/2}\right]\\
 &\hspace{0.5cm}\le 
 \mathbf{E}\left[u(x+S(m));\tau_x>m,
 \widetilde \nu_n> m,
 M(m)> n^{1/2+\varepsilon/2}\right]\\
 &\hspace{1cm}+
 C\mathbf{E}\left[|S(k)-S(m)|^p;\tau_x>m,
 M(m)> n^{1/2+\varepsilon/2}\right]\\
 &\hspace{1cm}+
 C{\bf 1}\{p\ge1\}\mathbf{E}\left[|S(k)-S(m)||x+S(m)|^{p-1};\tau_x>m,
 M(m)> n^{1/2+\varepsilon/2}\right]\\
 &\hspace{0.5cm}:=E_1+E_2+E_3.
\end{align*}
To estimate $E_1$ we apply the upper bound from \eqref{u-bounds}, 
and use the fact that on the event $\{\widetilde \nu_n>m\}$, 
$$
{\rm dist}(x+S(m),\partial K) \le \frac{1}{2}\left(
  n^{1/2-\varepsilon}+\frac{|x+S(m)|}{n^{2\varepsilon}} 
  \right).
$$
Then, 
\begin{align*}
  E_1&\le \frac{1}{2}n^{1/2-\varepsilon} 
  \mathbf{E}\left[(x+S(m))^{p-1};\tau_x>m,
 M(m)> n^{1/2+\varepsilon/2}\right]\\
 &\hspace{0.5cm}+
 \frac{1}{2n^{2\varepsilon}} 
  \mathbf{E}\left[(x+S(m))^{p};\tau_x>m,
 M(m)> n^{1/2+\varepsilon/2}\right].
\end{align*}  
Using independence of increments  we obtain 
$$
E_2\le Cn^{p/2}
\pr\left(\tau_x>m,
M(m)> n^{1/2+\varepsilon/2}\right)
$$
and 
$$
E_3\le Cn^{1/2} 
\mathbf{E}\left[(x+S(m))^{p-1};\tau_x>m,
 M(m)> n^{1/2+\varepsilon/2}\right].
$$
Combining these estimates and using the Markov inequality, we obtain 
\begin{align*}
  E_1+E_2+E_3\le 
  \frac{C}{n^{\min(p\varepsilon/2,\varepsilon/2)}} 
  \mathbf{E}\left[(M(m))^{p};\tau_x>m,
 M(m)> n^{1/2+\varepsilon/2}\right].
\end{align*}  
Now note that by Lemma~\ref{lem:tail}, 
\begin{equation}
  \label{nu.almost}
  \mathbf{E}\left[(M(m))^{p};\tau_x>m,
 M(m)> n^{1/2+\varepsilon/2}\right]\le C\mathbf{E}[\tau_x\wedge n].
\end{equation}  
Note that for $p>2$ the desired statement immediately follows from \eqref{MC}. 
If $p\le2$ then, using \eqref{MC}, 
$$
\mathbf{E}[\tau_x\wedge n]\le n^{1-p/2+\delta/2}\mathbf{E}[\tau_x^{p/2-\delta/2}]
\le C(1+|x|^{p-\delta})n^{1-p/2+\delta/2}.
$$
By the assumption $\mathbf{E}|X|^{2+\delta}<\infty$,
$$
\mathbf{E}\left[|X|^p;|X|>an^{1/2+\varepsilon/2}\right]
\le Cn^{-(1/2+\varepsilon/2)(2+\delta-p)}.
$$
Then using directly  the last inequality in the proof of Lemma~\eqref{lem:tail} we can see that 
 \eqref{nu.almost} remains valid for $p\le 2$.
The proof is complete. 
\end{proof}

\begin{remark}\label{without_mc}
  The only place we need to use the results of \cite{MC84} is the end of the last Lemma.
  To make the proof self-contained  we can use a different estimate 
  in \eqref{nu.almost}. Namely, we can directly use the estimate \eqref{eq.max.nag}
  with $t=p$ and then apply estimates $\e[\tau_x\wedge n]\le n$ and
  further assuming  that $\e|X|^{p+2}<\infty$ the Markov inequality to probability. 
  This would give the desired estimate in \eqref{nu.almost}. Thus we can avoid using
   the results of \cite{MC84} by imposing 2 additional moments.
\end{remark}  
\subsection{Proof of Theorem~\ref{thm:C2}}
Fix a large integer $n_0>0$ and 
put, for $m\ge 1$,
$$
n_m=[n_0^{((1-\varepsilon)^{-m})}],
$$
where $[r]$ denotes the integer part of $r$. Let $n$ be any integer. 
There exists unique $m$ such that $n\in (n_m,n_{m+1}]$. 
We first split the  expectation into two  parts,
\begin{align*}
&\mathbf{E}[u(x+S(n));\tau_x>n]=E_{1}(x)+E_{2}(x)\\
&\hspace{1cm}:=\mathbf{E}\left[u(x+S(n));\tau_x>n,\widetilde{\nu}_n\leq n_m\right]
+\mathbf{E}\left[u(x+S(n));\tau_x>n,\widetilde{\nu}_n> n_m\right].
\end{align*}
By Lemma~\ref{lem:nu}, since $n_m\ge n^{1-\varepsilon}$,
the second term on the right hand side is bounded by
\begin{align*}
E_{2}(x)\le\frac{C(x)}{n_m^q},
\end{align*}
where
$$
C(x)=C(1+|x|^{p-\gamma}).
$$
For the first term we have 
\begin{align*}
 E_{1}(x)&=\sum_{i=1}^{n_m}\int_{\widetilde{K}_{n,\varepsilon}}\mathbf P\{\widetilde{\nu}_n=i,\tau_x>i, x+S(i)\in dy\}
\mathbf{E}[u(y+S(n-i));\tau_y>n-i].
\end{align*}
Then, by Proposition~\ref{prop:iteration},
\begin{align*}
E_{1}(x)&\le \left(1+\frac{C}{n^q}\right)
\sum_{i=1}^{n_m}\int_{\widetilde{K}_{n,\varepsilon}}\mathbf P\{\widetilde{\nu}_n=i,\tau_x>i, x+S(i)\in dy\}u(y)\\
&\le \frac{\left(1+\frac{C}{n^q}\right)}{\left(1-\frac{C}{n_m^q}\right)}
\sum_{i=1}^{n_m}\int_{\widetilde{K}_{n,\varepsilon}}\mathbf P\{\widetilde{\nu}_n=i,\tau_x>i, x+S(i)\in dy\}\\
&\hspace{4cm}\times
\mathbf{E}[u(y+S(n_m-i));\tau_y>n_m-i]\\
&=\frac{\left(1+\frac{C}{n_m^q}\right)}{\left(1-\frac{C}{n_m^q}\right)}
\mathbf E[u(x+S(n_m));\tau_x>n_m,\widetilde{\nu}_n\le n_m].
\end{align*}
As a result we have
\begin{align}\label{L3.1}
\mathbf E[u(x+S(n));\tau_x>n]\leq
\frac{\left(1+\frac{C}{n_m^q}\right)}{\left(1-\frac{C}{n_m^q}\right)}
\mathbf E[u(x+S(n_m)); \tau_x>n_m]+\frac{C(x)}{n_m^q}.
\end{align}
Iterating this procedure $m$ times, we obtain
\begin{align}
\label{L3.2}
\nonumber
&\max_{n\in(n_m,n_{m+1}]}\mathbf{E}[u(x+S(n));\tau_x>n]\\
&\hspace{1cm}\leq
\prod_{j=0}^{m}\frac{\left(1+\frac{C}{n_j^q}\right)}{\left(1-\frac{C}{n_j^q}\right)}
\left(\mathbf E[u(x+S(n_0)); \tau_x>n_0]+C(x)\sum_{j=0}^{m}n_j^{-q}\right).
\end{align}
Since $n_m$ grows exponentially fast, we infer that 
\begin{equation}
\label{eq:sup}
\sup_n \mathbf E[u(x+S(n));\tau_x>n]\le C(x)<\infty. 
\end{equation}

An identical procedure gives a lower bound 
\begin{align}\label{L3.3}
\nonumber
&\mathbf E[u(x+S(n));\tau_x>n]\geq E_1(x)\\
\nonumber
&\geq\frac{\left(1-\frac{C}{n_m^q}\right)}{\left(1+\frac{C}{n_m^q}\right)}
\mathbf{E}[u(x+S(n_m));\tau_x>n_m,\widetilde{\nu}_{n}\leq n_m]\\
\nonumber
&\ge\frac{\left(1-\frac{C}{n_m^q}\right)}{\left(1+\frac{C}{n_m^q}\right)}
\Big(\mathbf{E}[u(x+S(n_m));\tau_x>n_m]-\mathbf{E}[u(x+S(n_m));\tau_x>n_m,\widetilde{\nu}_{n}>n_m]\Big)\\
\nonumber
&\ge\frac{\left(1-\frac{C}{n_m^q}\right)}{\left(1+\frac{C}{n_m^q}\right)}
\Big(\mathbf{E}[u(x+S(n_m));\tau_x>n_m]-C(x)n_m^{-q}\Big)\\
&\ge\prod_{j=0}^m \frac{\left(1-\frac{C}{n_j^q}\right)}{\left(1+\frac{C}{n_j^q}\right)}\mathbf{E}[u(x+S(n_0));\tau_x>n_0]
-C(x)\sum_{j=0}^mn_j^{-q}.
\end{align}
For every positive $\delta$ we can choose $n_0=n_0(\delta)$ such that 
for all $m\ge 1$, 
$$
\left|\prod_{j=0}^{m}\frac{\left(1-\frac{C}{n_j^q}\right)}{\left(1+\frac{C}{n_m^q}\right)}-1\right|\leq\delta
\quad\text{and}\quad \sum_{j=0}^{m}n_{j}^{-q}\leq\delta.
$$
Then, for this value of $n_0$ and all $x\in K$,
$$
\sup_{n>n_0}\mathbf E[u(x+S(n));\tau_x>n]
\leq (1+\delta)\mathbf E [u(x+S(n_0));\tau_x>n_0]+C(x)\delta
$$  
and 
\begin{equation}
  \label{eq.pos}
\inf_{n>n_0}\mathbf E [u(x+S(n));\tau_x>n]
\geq (1-\delta)\mathbf E  [u(x+S(n));\tau_x>n_0]-C(x)\delta. 
\end{equation}
Consequently,
\begin{align*}
\sup_{n>n_0}\mathbf E [u(x+S(n));\tau_x>n]
-\inf_{n>n_0}\mathbf E [u(x+S(n));\tau_x>n]\\
\leq 2\delta 
\mathbf E [u(x+S(n_0));\tau_x>n_0]+2C(x)\delta.
\end{align*}
Taking into account (\ref{eq:sup}) and  that $\delta$ can be made arbitrarily small we conclude that 
the limit 
$$
V(x):=\lim_{n\to\infty}\mathbf E [u(x+S(n));\tau_x>n]
$$
exists for every $x\in K$.

For positivity of $V$ note that 
by \eqref{diff-bound}, 
$$
\e[u(tx+S_{n_0});\tau_{tx}>n_0]
\ge u(tx)\pr(\tau_{tx}>n_0)
-\e[|S_{n_0}|^p]
-(t|x|)^{p-1}\e[|S_{n_0}|]
$$
when $p\ge 1$ and by \eqref{diff-bound2}, 
$$
\e[u(tx+S_{n_0});\tau_{tx}>n_0]
\ge u(tx)\pr(\tau_{tx}>n_0)
-\e[|S_{n_0}|^p]
$$
when $p<1$. 
Also, $C(tx)\le t^{p-\gamma}|x|^{p-\gamma}$. 
Hence, it follows from \eqref{eq.pos} that there exists $R$ such that $V(x)$ is positive for  $x\in D_{R,\gamma}$. The rest of the proof follows the corresponding part of Lemma~13 of \cite{DW15}. 
The proof is complete.

\end{document}